\documentclass[reqno, 12pt]{amsart}
\usepackage{amsmath,amscd,graphics,graphicx,color,a4wide,hyperref,verbatim}
\usepackage{mathptmx}
\usepackage{tikz, tikz-3dplot, pgfplots}
\usepackage{tkz-graph}
\usepackage{tikz-cd}
\usetikzlibrary[positioning,patterns] 
\usetikzlibrary{matrix,arrows,decorations.pathmorphing}

\usepackage{csquotes}

\usepackage{textcmds}
\usepackage{epic}
\usepackage{graphicx}
\usepackage{psfrag}
\usepackage{marvosym}
\usepackage{amsfonts}
\usepackage{amssymb}
\usepackage{mathrsfs}
\usepackage{enumitem}
\usepackage{subcaption}
\usepackage{ytableau}

\usepackage{amsthm,xcolor}

\usepackage{slashed}

\usepackage{dynkin-diagrams}

\usepackage{float}

\newtheorem{theorem}{Theorem}[section]
\newtheorem{lemma}[theorem]{Lemma}
\newtheorem{lem-def}[theorem]{Lemma-definition}
\newtheorem{proposition}[theorem]{Proposition}
\newtheorem{prop-def}[theorem]{Proposition-definition}

\newtheorem{conjecture}[theorem]{Conjecture}

\theoremstyle{definition}
\newtheorem{definition}[theorem]{Definition}

\newtheorem{example}[theorem]{Example}

\newtheorem{remark}[theorem]{Remark}

\newtheoremstyle{named}%
    {}{}{\itshape}{}{\bfseries}{.}{.5em}{\thmnote{#3}}
\theoremstyle{named}
\newtheorem*{namedtheorem}{Theorem}

\newcommand\bref[3][blue]{%
    \begingroup%
    \hypersetup{linkcolor=#1}%
    \hyperlink{#2}{#3}%
    \endgroup}

\numberwithin{equation}{section}

\newtheorem{thm}{Theorem}[section] 
\theoremstyle{plain} 
\newcommand{\thistheoremname}{}
\newtheorem{genericthm}[thm]{\thistheoremname}

\newtheorem*{genericthm*}{\thistheoremname}
\newenvironment{namedthm*}[1]
  {\renewcommand{\thistheoremname}{#1}%
   \begin{genericthm*}}
  {\end{genericthm*}}

\newcommand{\CC} {\mathbb{C}}

\newcommand{\LL} {\mathbb{L}}

\newcommand{\PP} {\mathbb{P}}

\newcommand{\RR} {\mathbb{R}}

\newcommand{\ZZ} {\mathbb{Z}}

\newcommand {\shA} {\mathcal{A}}
\newcommand {\shB} {\mathcal{B}}

\newcommand {\shD} {\mathcal{D}}

\newcommand {\shO} {\mathcal{O}}

\newcommand {\bR}{\mathbf{R}}

\newcommand {\Cone} {\operatorname{Cone}}

\newcommand {\Ext} {\operatorname{Ext}}
\newcommand{\sExt}{\mathscr{E} \kern -3pt xt}

\newcommand {\Fl} {\operatorname{Fl}}

\newcommand {\Hom} {\operatorname{Hom}}
\newcommand {\sHom}{\mathscr{H}\kern-5pt\mathcalligra{om}}

\renewcommand {\Im} {\operatorname{Im}}

\newcommand {\Ker} {\operatorname{Ker}}

\newcommand {\Pic} {\operatorname{Pic}}

\newcommand {\rank} {\operatorname{rank}}

\newcommand {\Tot} {\operatorname{Tot}}
\newcommand{\IGr}{\operatorname{IGr}}

\newcommand {\IFl} {\operatorname{IFl}}

\usepackage{tablefootnote}

\title[]{On Derived Categories of Generalized Grassmannian Flips}
\author{Naichung Conan Leung}
\address{The Institute of Mathematical Sciences\\The Chinese University of Hong Kong\\Shatin, N.T.\\ Hong Kong}
\email{leung@math.cuhk.edu.hk}

\author{Ying Xie}
\address{Southern University of Science and Technology\\Shenzhen\\China}
\email{xiey@sustech.edu.cn}
\begin{document}

\begin{abstract}
In this paper, we construct and classify a new family of flips, called generalized Grassmannian flips, by generalizing the construction of standard flips for $\PP^m\times \PP^n$ to any generalized Grassmannian $G/P$, where $P$ is a maximal parabolic subgroup of a complex semi-simple algebraic group. In addition, we show that a 9-fold generalized Grassmannian flip for $Sp(6, \CC)$ satisfies the DK flip conjecture by Bondal-Orlov \cite{bondal2002derived} and Kawamata \cite{kawamata2002d} via mutation techniques by Kuznetsov and Thomas' chess game method. 
\end{abstract}

\maketitle

\section{Introduction}
Flip is a fundamental surgery operation for constructing minimal models in high-dimensional birational geometry. Minimal models are also expected to be minimal with respect to derived categories. Bondal-Orlov \cite{bondal2002derived} and Kawamata \cite{kawamata2002d} conjectured that for any flip between smooth projective varieties $Y_2 \dashrightarrow Y_1$, there is a derived embedding on their bounded derived categories of coherent sheaves $D^b Coh(Y_1) \hookrightarrow D^b Coh(Y_2)$. This conjecture (DK Flip Conjecture throughout this paper) has been verified for some toroidal flips\footnote{See \cite{kawamata20054}, \cite{kawamata2006}, \cite{kawamata2013} and \cite{Kawamata2016}.} and \emph{Grassmannian flips}\footnote{See \cite{aflip}.}.

In the first part of this paper, we investigate \emph{simple flips} (Definition \ref{defsimpleflip}) and classify simple flips of homogeneous type of rank 1, i.e., \emph{generalized Grassmannian flips} (Definition \ref{defgflip}),  whose exceptional sets are generalized Grassmannians. Recall that a homogeneous variety $G/P$ is a generalized Grassmannian if $P$ is a maximal parabolic subgroup of a complex semi-simple algebraic group $G$. When $G=SL(n, \CC)$, generalized Grassmannians are the ordinary Grassmannians. Note that high-rank (rank$\geq 2$)  simple flips of homogeneous type can be viewed as families of generalized Grassmannian flips (see Example \ref{familyC3}).


 \begin{theorem}[(=Theorem \ref{main2})]
 There is a classification of generalized Grassmannian flips via marked Dynkin diagrams. (See Appendix for labelings of Dynkin diagrams.)
  \end{theorem}

 \begin{table}[H]
\begin{center}
\resizebox{\textwidth}{!}{
\begin{tabular}{|c|c|}
\hline
Dynkin diagram   &generalized Grassmannian flips \\
\hline
$A_m \times A_n\, (m> n)$\tablefootnote{It is a flop when $n=m$.}  &$(\alpha_1, \beta_1)$         \\
$A_n\, (n\geq 2)$  &$(\alpha_{i+1}, \alpha_{i})$ $(1\leq i<\dfrac{n}{2})$; $(\alpha_i, \alpha_{i+1})$ $(\dfrac{n}{2}<i<n)$ \tablefootnote{It is a flop when $i=n/2$.} \\
$B_n\,(n \geq 3)$   &  $(\alpha_{n-1}, \alpha_{n})$ for all $n$ and $(\alpha_3, \alpha_1)$ for $B_3$  \\
$C_n\, (n\geq 3)$ &  $(\alpha_{i+1}, \alpha_i) (1\leq i<\dfrac{2n-1}{3})$; $(\alpha_{i}, \alpha_{i+1}) (\dfrac{2n-1}{3}<i<n)$ \tablefootnote{It is a flop when $i=(2n-1)/3$.}\\
$D_n\, (n\geq 4), E_6, E_7, E_8, F_4, G_2$ & None \tablefootnote{We only consider $n>3$ here since $D_2=A_1\times A_1, D_3=A_3$.}\\
\hline
\end{tabular}}\caption{Classification of generalized Grassmannian flips}\label{cgflip}
\end{center}
\end{table}
 
 \begin{remark}
 \begin{enumerate}
 \item The above classification is based on \cite{kanemitsu2022mukai} which explicitly classifies all simple \emph{flops} of homogeneous type.
 \item The generalized Grassmannian flips $(A_m \times A_n, \alpha_1, \beta_1)$ in the first row are the standard flips, and the generalized Grassmannian flips $(A_n, \alpha_2, \alpha_1)$ in the second row are the Grassmannian flips in \cite{aflip}.
    \end{enumerate}
 \end{remark}

In the second part of this paper, we verify the DK Flip Conjecture for the generalized Grassmannian flip $(C_3, \alpha_2, \alpha_3)$ in detail. 

We consider the isotropic Grassmannians $\IGr(i, 6)$, which consists of all $i$-th dimensional isotropic linear  subspaces of the symplectic vector space $(\CC^6, \omega)$.  

Let $h_i$ denote the ample generator of $\Pic(\IGr(i, 6))$ for $i=2, 3$. The tautological bundle over $\IGr(i, 6)$ throughout this paper, denoted as $U_i$, is a subbundle of the trivial bundle $\underline{\CC^6}=\CC^6\otimes \shO$. The isotropic condition 
gives a bundle morphism $\underline{\CC^6}/U_2 \xrightarrow{\lrcorner\omega}U_2^{\vee}$, which is defined as the contraction by the linear symplectic form $\omega$. Denote its kernel by $N_{\omega}$, and consider    
\begin{align*}
\quad X_2=\PP_{\IGr(2, 6)}(N_{\omega}(-2h_2)\oplus \shO), \quad X_3=\PP_{\IGr(3, 6)}(U_3^{\vee}(-2h_3)\oplus \shO).
\end{align*}

Then there is a generalized Grassmannian flip:
\begin{equation}\label{gflips}
\begin{tikzcd}
X_2 \arrow[r, dashed, "g"] & X_3.
\end{tikzcd}
\end{equation}

\begin{theorem}\label{thm}
For the flip $g$ (\ref{gflips}), there is an embedding of their derived categories of coherent sheaves:
\begin{center}
\begin{tikzcd}
D^b Coh(X_3) \arrow[r, hook] & D^b Coh(X_2).
\end{tikzcd}
\end{center}
\end{theorem}

Let us outline the proof of Theorem \ref{thm}. Firstly, $g$ can be realized as blowing-up $X_2$ along $\IGr(2, 6)$ and subsequently blowing down it onto $X_3$ with the center $\IGr(3, 6)$, i.e., there is an isomorphism
$$ Bl_{\IGr(2, 6)} X_2\cong Bl_{\IGr(3, 6)} X_3=:X.$$
 
Secondly, Orlov's formula on derived categories of blowing-ups gives two semi-orthogonal decompositions (SOD) of $D(X):=D^b Coh(X)$:
\begin{align}
&D(X)=\langle D(X_2), D(\IGr(2, 6))\rangle; \label{sod1}\\
&D(X)=\langle D(X_3), D(\IGr(3, 6)), D(\IGr(3, 6))\otimes\shO(h_2)\rangle. \label{sod2}
\end{align}

Thirdly, in order to construct an embedding of $D(X_3)\hookrightarrow D(X_2)$, it is natural to consider the functor $F: D(X_3) \hookrightarrow D(X) \rightarrow D(X_2)$ defined by the composition of the pullback and the pushforward of the blowing-up morphisms $X\rightarrow X_3$ and $X\rightarrow X_2$, respectively. If $g$ were a toroidal flip, the results of Kawamata (\cite{kawamata20054}, \cite{kawamata2006}, \cite{kawamata2013} and \cite{Kawamata2016}) would imply that $F$ does give the desired embedding. But unfortunately, since the flip $g$ is not toroidal, the functor $F$ is usually not fully-faithful. We can therefore apply the mutation techniques by Kuznetsov (\cite{kuznetsov2010derived}) to the SODs (\ref{sod1}) and (\ref{sod2}) to get new SODs $D(X)=\langle \shD_2, {}^{\perp}\shD_2\rangle$ (\ref{sod1'}) and $D(X)=\langle \shD_3, {}^{\perp}\shD_3\rangle$ (\ref{sod2'}), where  $\shD_i\cong D(X_i), i=2, 3$. In the new SODs, the subcategories ${}^{\perp}\shD_2$ and ${}^{\perp}\shD_3$ have more common objects than those in (\ref{sod1}) and (\ref{sod2}). 

Fourthly, by applying Thomas' chess game method (\cite{thomas2018notes})\footnote{See \cite{jiang2021categorical} and \cite{jiang2018derivedcat}, \cite{jiang2018blowing}, \cite{jiang2018categorical}, \cite{jiang2022derived} for other applications of this method.}, we show that the natural functor $\phi: \shD_3  \hookrightarrow D(X) \rightarrow \shD_2$ is fully-faithful, which yields a fully-faithful embedding $D(X_3)\hookrightarrow D(X_2)$(Theorem \ref{thm}).

The rest of this paper is organized as follows. In section 2, we introduce the notion of generalized Grassmannian flips and give a classification of all possible local models. Section 3 covers all necessary mutation calculations. Section 4 presents the detailed proof of Theorem \ref{thm}.

\subsection*{Conventions}
In this paper, $\PP(F)=Proj (Sym^{\bullet}F^{\vee})$ for any vector bundle $F$, and $D(S)=D^b Coh(S)$ is the bounded derived category of coherent sheaves on a smooth projective variety $S$.  We write $E_1\sim E_2$ if two divisors $E_1$ and $E_2$ are linearly equivalent. The derived functors $\RR Hom(-,-)$ and $\Ext^{\bullet}(-,-)$ are taken over the total space $X$. We will omit the natural functors $p_1^*, p_2^*$ and $j_*$ unless they are not obvious in the content. The ground field of all varieties is the complex number field $\CC$.

\section{Generalized Grassmannian Flips}
\subsection{Structures of Simple Flips}
 
\begin{definition}[(\cite{kollar1999rational}, \cite{kollar2008birational}, \cite{debarre2013higher})]
Let $Y_1$ and $Y_2$ be two smooth projective varieties. A birational map $f: Y_2\dashrightarrow Y_1$ is called a \emph{flip} if there is a normal variety $W$ with two small contractions\footnote{I.e., isomorphisms in codimension 1.} $\phi_i: Y_i\rightarrow W$ $(i=1, 2)$  such that
\begin{enumerate}
\item $f=\phi_1^{-1}\circ \phi_2$. 
\item $\phi_i$ is elementary, i.e., the relative Picard number $\rho(Y_i/W)=\rho(Y_i)-\rho(W)=1$ $(i=1, 2)$.
\item $-K_{Y_2}$ is $\phi_2$-ample and $K_{Y_1}$ is $\phi_1$-ample.  
\end{enumerate}
\end{definition}

For any birational map $f$, one can always find a smooth model $X$ by resolving the indeterminacy of $f$ via taking a resolution of singularities of the graph closure $\overline{\Gamma_{f}}\subset Y_2\times Y_1$.
That is, $f=\pi_1\circ\pi_2^{-1}$, where $\pi_i$'s are birational morphisms.
\begin{equation}
 \begin{tikzcd}
& \arrow[dl,swap, "\pi_2"]  X \arrow[dr, "\pi_1"]\\
   Y_2 \arrow[rr,dashed, "f"]     && Y_1
\end{tikzcd}
\end{equation} 

\begin{definition}\label{defsimpleflip}
A flip $f$ is called \emph{simple}, if both $\pi_1$ and $\pi_2$ are blowing-ups along certain smooth centers $Z_1$ and $Z_2$, respectively. 
\end{definition}

\begin{proposition}\label{exp}
Let $E_i$ be the exceptional divisor of $\pi_i$ $(i=1, 2)$. Then we have 
(i) $E_1=E_2$; (ii) $Y_2>_K Y_1$, i.e. $\pi_2^*K_{Y_2}-\pi_1^*K_{Y_1}$ is linearly equivalent to an effective divisor. 
\begin{equation}
 \begin{tikzcd}
E_2 \arrow[d,swap,"p_2"] \arrow[r, hook]   & X \arrow[d,swap, "\pi_2"] \arrow[r, equal]  &X  \arrow[d, "\pi_1"] &\arrow[l,hook']  E_1\arrow[d,"p_1"] \\
   Z_2\arrow[r,hook] & Y_2 \arrow[r,dashed, "f"]      &  Y_1& \arrow[l, hook',swap]  Z_1
        \end{tikzcd}
\end{equation}
\end{proposition}
\begin{proof}
We follow the argument in the appendix of \cite{wang2002cohomology} and Lemma 4.4 in \cite{kawamata2002d}. Let $r_i$ be the codimension of $Z_i$ in $Y_i$ $(i=1, 2)$. By the formula for the canonical divisor of a smooth blowing-up, we have
\begin{align*}
&K_{X}\sim\pi_1^*K_{Y_1}+ (r_1-1) E_1, \\
&K_X\sim\pi_2^*K_{Y_2}+ (r_2-1) E_2.
\end{align*}
Assume that $E_1\neq E_2$, and let $H_2$ and $H$ be ample divisors on $Y_2$ and $X$, respectively, and let $n=\dim X$. Since $E_2$ is an exceptional divisor, we have
$$\pi_2^*H_2^{(n-r_2)}\cdot H^{r_2-2}\cdot (r_2-1)E_2^2<0.$$ Meanwhile, 
$$\pi_2^*H_2^{(n-r_2)}\cdot H^{r_2-2}\cdot (r_2-1) E_2^2=\pi_2^*H_2^{(n-r_2)}\cdot H^{r_2-2}\cdot E_2 \cdot (\pi_1^*K_{Y_1}+ (r_1-1) E_1-\pi_2^*K_{Y_2}).$$
Note also that $\pi_2^*H_2^{(n-r_2)}\cdot H^{r_2-2}\cdot E_2\cdot  \pi_2^*K_{Y_2}=0$ and 
$ \pi_2^*H_2^{(n-r_2)}\cdot H^{r_2-2}\cdot E_2\cdot (r_1-1) E_1\geq 0$.  
Therefore
$$\pi_2^*H_2^{(n-r_2)}\cdot H^{r_2-2}\cdot (r_2-1) E_2^2\geq \pi_2^*H_2^{(n-r_2)}\cdot H^{r_2-2}\cdot E_2 \cdot \pi_1^* K_{Y_1}> 0$$
since $K_{Y_1}$ is $\phi_1$-ample. A contradiction arises and hence, 
\begin{align}
E_1=E_2=:E\quad \text{and}\quad \pi_2^*K_{X_2}-\pi_1^*K_{X_1}\sim(r_1-r_2) E. \label{eq:ex}
\end{align}
By taking the intersection product with $(\pi_2^*H_2^{(n-r_2)}\cdot H^{r_2-2}\cdot E)$ on both sides of (\ref{eq:ex}), we have
$$\pi_2^*H_2^{(n-r_2)}\cdot H^{r_2-2}\cdot E\cdot (\pi_2^*K_{X_2}-\pi_1^*K_{X_1})<0.$$ 
It follows intermediately that 
$$\pi_2^*H_2^{(n-r_2)}\cdot H^{r_2-2}\cdot E\cdot (r_1-r_2) E<0.$$ 
Consequently, $r_1>r_2$ and $Y_2>_KY_1$. 

\end{proof}

We can therefore summarize the geometry of a simple flip $f$ in the following diagram:
\begin{equation}\label{sflip2}
 \begin{tikzcd}
&&E \arrow[ddll,swap,"p_2"] \arrow[d, hook]\arrow[ddrr,"p_1"] \\
&&  \arrow[dl,swap, "\pi_2"]  X \arrow[dr, "\pi_1"]\\
   Z_2\arrow[r,hook] \arrow[ddrr, swap, "q_2"] & Y_2 \arrow[rr,dashed, "f"]  \arrow[dr, swap, "\phi_2"]   & &  \arrow[dl, "\phi_1"] Y_1& \arrow[l, hook',swap]  \arrow[ddll, "q_1"] Z_1.\\
        & & W \\
        && B \arrow[u, hook]
\end{tikzcd}
\end{equation}

To achieve such a simple flip (\ref{sflip2}), the following necessary conditions must be satisfied:

\begin{enumerate}
\item[(1)]
$E$ admits two smooth projective bundle structures given by $p_1$ and $p_2$.
\item[(2)]
By Corollary 2.6 of Kanemitsu \cite{kanemitsu2022mukai}, the contraction $q_1\circ p_1=q_2\circ p_2: E\rightarrow B$ is smooth, and each fiber is a Fano variety of Picard number two. 
\item[(3)]
$\shO_E(E)$ is the tautological line bundles for the projective bundle $p_1$, as well as $p_2$. 
\end{enumerate}

\subsection{Construction of Simple Flips}
We start with a smooth projective variety $Y_2$ and a smooth subvariety $Z_2$ of codimension $r_2$ ($r_2 \geq 2$). Let $X$ be the blowing-up of $Y_2$ along $Z_2$ with the exceptional divisor $E$, and let $p_2$ be the contraction of $E$ onto $Z_2$. 

\begin{lemma}\label{kineq}
Assume the following conditions hold:
\begin{enumerate}
\item[(1)] $E$ admits another smooth projective bundle structure $p_1: E\rightarrow Z_1$, where $Z_1$ is smooth projective;
\item[(2)] $r_1>r_2$, where $r_i-1$ is the dimension of fibers of $p_i$ ($i=1, 2$);
\item[(3)] the restriction of $\shO(-E)$ to each fiber $\PP^{r_1-1}$ of $p_1$ is $\shO(1)$. 
\end{enumerate} 
Then we have the following:
\begin{itemize}

\item[(i)] The extremal rays of $p_1$ and $p_2$ form an extremal face of $E$. Moreover, there is a smooth extremal contraction from $E$ onto $B$ such that the following diagram is commutative:

\begin{equation}
\begin{tikzcd}
& E\arrow[dl, swap, "p_2"] \arrow [dr, "p_1"]\\
Z_2\arrow[dr, swap, "q_2"]&&Z_1 \arrow [dl, "q_1"] \\
& B
\end{tikzcd}.
\end{equation}
\item[(ii)]
There is a blowing-down morphism $\pi_1$ from $X$ onto a smooth complex manifold $Y_1$ that is the smooth blowing-up of $Y_1$ along $Z_1$ and covers $p_1$. 

\item[(iii)]
If there is a birational contraction $\phi_2$ from  $Y_2$ onto $W$ with the exceptional locus $Z_2$ that covers the extremal contraction $q_2$, then $Y_1$ is projective and the map (\ref{sflip2}) $f: Y_2\dashrightarrow Y_1$ is a simple flip. 
\end{itemize}
\end{lemma}
\begin{proof}
Claim (i) follows from Corollary 2.6 of Kanemitsu \cite{kanemitsu2022mukai}. Note that only condition (1) is used here. Claim (ii) follows from the blowing-down criterion by Nakano \cite{nakano1971inverse} and Fujiki \cite{fujiki1971supplement}. 

For claim (iii), by the formula for the canonical divisor $K_X$ of the two smooth blowing-ups $\pi_1$ and $\pi_2$, we have
\begin{align}
 \pi_2^*K_{Y_2}-\pi_1^*K_{Y_1}\sim(r_1-r_2) E. \label{adj}
 \end{align}
Let $Z_{ib}$ denote the fiber of $q_i$ over $b$ for each $b\in B$. Restricting (\ref{adj}) to $b$, we have
\begin{align}
\shO(p_2^*K_{Y_2}|_{Z_{2b}}-p_1^*K_{Y_1}|_{Z_{1b}})=\shO_{E_b}(-1)^{\otimes{(r_1-r_2)}}.\label{adj2}
\end{align}
If we restrict (\ref{adj2}) further to a fiber $\PP^{r_1-1}$ of $p_1$, then
\begin{align}
 \shO(p_2^*K_{Y_2}|_{\PP^{r_1-1}})=\shO_{\PP^{r_1-1}}(r_2-r_1).\label{adj3}
 \end{align}
Note that $r_1>r_2$ (condition (2)), and $Z_{2b}$ is a Fano variety of Picard number 1 by Corollary 2.6 of Kanemitsu \cite{kanemitsu2022mukai}. Therefore, $K_{Y_2}|_{Z_{2b}}$ must be anti-ample. Similarly, $K_{Y_1}|_{Z_{1b}}$ must be ample. Thus, $\phi_2$ is a flip contraction. 

Let $C_1$ be a line that generates the extremal ray of $p_1$. It is easy to obtain that 
$$C_1\cdot K_X=C_1\cdot \pi_2^*K_{Y_2}+C_1\cdot (r_2-1)E<0,$$ 
since $-K_{Y_2}$ is $\phi_2$-ample and $C_1\cdot E<0$ (condition (3)). Hence, $C_1$ is $K_X$-negative.  The uniqueness of the contraction map (Theorem 7.39 and Lemma 1.15 in \cite{debarre2013higher}) implies that $\pi_1: X\rightarrow Y_1$ is the contraction map of $C_1$. Then $Y_1$ is projective by construction. By Lemma 1.15 in \cite{debarre2013higher}, there exists $\phi_1: Y_1\rightarrow W$ such that the contraction $\phi_2\circ \pi_2: X\rightarrow W$ factors through $\pi_1: X\rightarrow Y_1$. Therefore, $f=\phi_1 \circ \phi_2^{-1}$ is a simple flip.

\end{proof}

\subsection{Simple Flips of homogeneous Type}
In this subsection, we focus on \emph{simple flips of homogeneous type}. 

\begin{definition}\label{defgflip}
A simple flip $f$ is called of \emph{homogeneous type} if the exceptional divisor $E$ in Proposition \ref{exp} is a homogeneous variety. If the Picard number $\rho(E)=r+1$, then $f$ is called of \emph{rank} $r$. Moreover, $f$ is called a \emph{generalized Grassmannian flip} if $r=1$.   
\end{definition}

If $G$ is a complex semi-simple algebraic group, then any $G$-homogeneous variety $M$ is in one-to-one correspondence to a \emph{marked Dynkin diagram}, \emph{i.e.}, a (possibly disconnected) Dynkin diagram $\Gamma$ with a set of nodes $I$ called \emph{marking} (\cite{fulton2013representation}). Specifically, $M=G/P_I$, in which $G$ corresponds to $\Gamma$, and the parabolic subgroup $P_I$ of $G$   corresponds to the marking $I$. For any sub-marking $J\subset I$, there is an inclusion of parabolic subgroups $P_I\subset P_J$ that gives a projective morphism:
$$p_{I, J}:\, G/P_I\longrightarrow G/P_J,$$
and each fiber of $p_{I,J}$ is also a homogeneous variety (isomorphic to $P_J/P_I$). 

When $G$ is simple, \emph{i.e.}, $\Gamma$ is connected, $P_J/P_I$ corresponds to a marked Dynkin diagram called \emph{Tits shadow}. It is   
a diagram with the marking $I\backslash J$ obtained by removing from $\Gamma$ all the nodes in $J$ and the edges to which those nodes are connected (see \cite{tits1954groupes} and \cite{landsberg2003projective}). Furthermore, it is not hard to generalize this principle to semi-simple,  but not necessarily simple groups.

\begin{example}\label{C3fl}
We consider $G=Sp(6, \CC)$ and the Dynkin diagram of type $C_3$ with the marking  $I=\{\alpha_2, \alpha_3\}$ 
\begin{equation}\label{C3}
 \dynkin[edge length=1.5, root radius=.1cm, labels={\alpha_1, \alpha_2, \alpha_3}]{C}{*XX}.
\end{equation}
Then the corresponding homogeneous varieties are
\begin{align*}
&G/P_{\alpha_1\cup\alpha_2}=\IFl(1, 2, \CC^6)\\
&=\{(V_1, V_2)|\, V_1\leq V_2\leq \CC^6, \,\text{dim } V_1=1, \text{dim } V_2=2, \omega|_{V_2}=0\},
\end{align*}
$$G/P_{\alpha_2}=\IGr(2, 6), \quad G/P_{\alpha_3}=\IGr(3, 6).$$
By means of Tits shadow,  the general fiber of $p_3: \IFl(1, 2, \CC^6)\rightarrow \IGr(3, 6)$ is isomorphic to 
\begin{equation*}
 \dynkin[edge length=1.5, root radius=.1cm, labels={\alpha_1, \alpha_2}]{A}{*X}\longleftrightarrow \PP^2, 
  \end{equation*}
 and the general fiber of $p_2: \IFl(1, 2, \CC^6)\rightarrow \IGr(2, 6)$ is $\PP^1$ is isomorphic to
 \begin{equation*}
 \dynkin[edge length=1.5, root radius=.1cm, labels={ \alpha_3}]{A}{X}\longleftrightarrow \PP^1.
 \end{equation*}
 In fact,
$$\Fl(2,3,6)=\PP_{\IGr(3,6)}(U_3^{\vee})=\PP_{\IGr(2,6)}(N_{\omega}).$$
\end{example}

Let ($\Gamma, I)$ be a marked Dynkin diagram which corresponds to a complex semi-simple algebraic group $G$ and a parabolic subgroup $P_I$. Consider two different submarkings $I_{\alpha}$ and $I_{\beta}$ corresponding to two parabolic subgroups $P_{I_{\alpha}}\supset P_I $ and $P_{I_{\beta}}\supset P_I$, respectively, and the two natural projections
$$p_{\alpha}: G/P_I\longrightarrow G/P_{I_{\alpha}},\qquad p_{\beta}: G/P_I\longrightarrow G/P_{I_{\beta}}.$$
By Lemma \ref{kineq}, we need to find such a variety $E$ with two different projective bundle structures. In other words, we will need the following condition:
\begin{namedtheorem}[\hypertarget{thm:asp}{Condition (A)}]
The general fibers of both $p_{\alpha}$ and $p_{\beta}$ are projective spaces. 
\end{namedtheorem}
Since the relative Picard number $\rho(p_i)=1$, $I\backslash I_{i}$ consists of a single node for $i=\alpha, \beta$. Denote $I\backslash I_{\alpha}$ by $\{\beta\}$, $I\backslash I_{\beta}$ by $\{\alpha\}$ and $I_{\alpha}\cap I_{\beta}$ by $J$. Then $I_{\alpha}= J\cup\{\alpha\}, 
I_{\beta}= J\cup\{\beta\}$ and $I=J\cup\{\alpha\}\cup\{\beta\}$. Consider the following diagram:
\begin{equation}\label{basic}
 \begin{tikzcd}
 &\arrow[dl, swap,"p_{\alpha}"]  G/P_I\arrow[dr, "p_{\beta}"]\\
G/P_{I_{\alpha}}\arrow[dr, swap, "q_{\alpha}"] && G/P_{I_{\beta}}\arrow[dl, "q_{\beta}"]\\
& G/P_J
\end{tikzcd}.
\end{equation}
Note that $\rho(q_i)=1$, and we denote by $h_i$ the ample generator of the relative Picard group of $q_i$. 

 \begin{lemma}\label{2pj}
Suppose that $(\Gamma, I)$ satisfies \bref{thm:asp}{Condition (A)}, and let $E_i=p_{i*}(\shO(p_{\alpha}^*h_{\alpha}+p_{\beta}^*h_{\beta}))$ for $i=\alpha, \beta$. Then we have
\begin{enumerate}
\item $E_i$ is a vector bundle on $G/P_{I_i}$, and 
\item $p_i$ is isomorphic to the projective bundle associated to $E_i^{\vee}$, i.e., $G/P_I=\PP(E_{i}^{\vee})$. 
\end{enumerate}
\end{lemma}
 
\begin{proof}
Let $p=q_{\alpha}\circ p_{\alpha}$. By restricting diagram (\ref{basic}) to the fiber of $p$ over a point $x\in G/P_J$, we may assume that $I$ consists of $\{\alpha, \beta\}$, that $h_i$ is the ample generator of the Picard group of the generalized Grassmannian $G/P_i$, and that $p_i$ is the projection $G/P_{\alpha\cup \beta}\rightarrow G/P_i$. By \bref{thm:asp}{Condition (A)},  $\forall y\in G/P_{\alpha},\, p_{\alpha}^{-1}(y)\cong \PP^{r_{\alpha}-1}$ for some $r_{\alpha}\in \ZZ$. Moreover, $p_{\beta}$ maps the fiber $p_{\alpha}^{-1}(y)$ isomorphically onto its image in $G/P_{\beta}$. Then $\shO(-p_{\alpha}^*h_{\alpha}-p_{\beta}^*h_{\beta})|_{p_{\alpha}^{-1}(y)}=\shO(-h_\beta)|_{p_{\beta}(p_{\alpha}^{-1}(y))}=\shO_{\PP^{r_{\alpha}-1}}(-1)$.  That is, $\shO(p_{\alpha}^*h_{\alpha}+p_{\beta}^*h_{\beta})$ is just the relative tautological dual line bundle on the projective fibration $p_{\alpha}$, which means $E_{\alpha}^{\vee}$ is a vector bundle, and $p_{\alpha}$ is the projective bundle associated to it. By repeating the same argument to $p_{\beta}$, we can complete the proof.
\end{proof}

Let $X_i=\PP_{G/ P_{I_i}}(E_i^{\vee}\oplus \shO)$ ($i=\alpha, \beta$) and $Z_i$ be the closed subvariety induced from the canonical section  $\shO\hookrightarrow E_i^{\vee}\oplus \shO$. Note that $Z_i\cong G/P_{I_i}$.  Then $Bl_{Z_i} X_i\cong \PP_{\PP(E_i^{\vee})}(\shO(\xi_i)\oplus\shO):=X$, where $\shO(\xi_i)$ is the relative tautological line bundle of $p_i$. By Lemma \ref{2pj},  $\shO(\xi_i)=\shO(-p_{\alpha}^*h_{\alpha}-p_{\beta}^*h_{\beta})$. Taking together with the proof of Lemma \ref{kineq}, we have the following diagram:
 
\begin{equation}\label{pregflip}
 \begin{tikzcd}
&&E=G/P_{I} \arrow[ddll,swap,"p_{\alpha}"] \arrow[d, hook, "j"]\arrow[ddrr,"p_{\beta}"] \\
&&  \arrow[dl,swap, "\pi_{\alpha}"]  X \arrow[dr, "\pi_{\beta}"]\\
   G/P_{I_{\alpha}}\arrow[r,hook, "i_{\alpha}"] \arrow[drr, swap, "q_{\alpha}"] & X_{\alpha} \arrow[rr,dashed, "f"]   & &  X_{\beta}& \arrow[l, hook',swap, "i_{\beta}"]  \arrow[dll, "q_{\beta}"] G/P_{I_{\beta}},\\
        && G/P_J
\end{tikzcd}
\end{equation}
where $E$ is the exceptional divisor of the two blowing-ups.

\begin{lemma}\label{flip}
The birational map $f: X_{\alpha}\dashrightarrow X_{\beta}$ is a simple flip if $\dim\, G/P_{I_{\beta}}< \dim\, G/P_{I_{\alpha}}$. 
\end{lemma}
\begin{proof}
Let $r_{\alpha}$ and $r_{\beta}$ be the codimensions of $i_{\alpha}$ and $i_{\beta}$, respectively. Note that $\dim E= \dim G/P_{I_{\alpha}} +r_{\alpha}-1=\dim G/P_{I_{\beta}} +r_{\beta}-1$.  Then $r_{\alpha}<r_{\beta}$ holds if and only if $\dim\, G/P_{I_{\beta}}< \dim\, G/P_{I_{\alpha}}$. To conclude Lemma \ref{flip}, it is sufficient to show that $r_{\alpha}<r_{\beta}$ and that there is an extremal contraction $\phi_{\alpha}$ of $X_{\alpha}$ onto some $W$ by contracting $G/P_{I_{\alpha}}$ onto $G/P_J$ according to Lemma \ref{kineq}.

To construct the contraction $\phi_{\alpha}$ by the contraction theorem, it is sufficient to show that the extremal ray of $q_{\alpha}$ is $K_{X_{\alpha}}$-negative. Let $C$ be a line in $G/P_{I_{\alpha}}$ that generates the extremal ray of $q_{\alpha}$. By the canonical divisor formula for two blowing-ups $\pi_{\alpha}$ and $\pi_{\beta}$, we have
\begin{align}
 \pi_{\alpha}^*K_{\alpha}-\pi_{\beta}^*K_{\beta}\sim(r_{\beta}-r_{\alpha})E. \label{res2}
 \end{align}
Restricting (\ref{res2}) to the fiber $E_b$ of $E\rightarrow B=G/P_J$ for $b\in B$, we obtain
$$ p_{\alpha}^*K_{\alpha}|_{q_{\alpha}^{-1}(b)}-p_{\beta}^*K_{\beta}|_{q_{\beta}^{-1}(b)}\sim(r_{\beta}-r_{\alpha})(-h_{\alpha}-h_{\beta}).$$
Note that $\Pic(q_{\alpha}^{-1}(b)) = \ZZ h_{\alpha}, \Pic(q_{\beta}^{-1}(b)) = \ZZ h_{\beta}$, and $E$ is of Picard rank 2 generated by $h_{\alpha}$
and $h_{\beta}$. Thus,
$$ K_{\alpha}|_{q_{\alpha}^{-1}(b)}\sim(r_{\beta}-r_{\alpha})(-h_{\alpha})\quad\text{and}\quad K_{\beta}|_{q_{\beta}^{-1}(b)}\sim(r_{\beta}-r_{\alpha})h_{\beta}.$$
Hence, $K_{X_{\alpha}}\cdot C=-(r_{\beta}-r_{\alpha})<0$, and the desired extremal contraction is obtained by contraction theorem. Therefore, $f$ is a simple flip as shown in the following diagram:

\begin{equation}\label{gflip}
 \begin{tikzcd}
&&E=G/P_{I} \arrow[ddll,swap,"p_{\alpha}"] \arrow[d, hook, "j"]\arrow[ddrr,"p_{\beta}"] \\
&&  \arrow[dl,swap, "\pi_{\alpha}"]  X \arrow[dr, "\pi_{\beta}"]\\
   G/P_{I_{\alpha}}\arrow[r,hook, "i_{\alpha}"] \arrow[ddrr, swap, "q_{\alpha}"] & X_{\alpha} \arrow[rr,dashed, "f"]  \arrow[dr, swap, "\phi_{\alpha}"]   & &  \arrow[dl, "\phi_{\beta}"] X_{\beta}& \arrow[l, hook',swap, "i_{\beta}"]  \arrow[ddll, "q_{\beta}"] G/P_{I_{\beta}}.\\
        & & W \\
        && B=G/P_J \arrow[u, hook]
\end{tikzcd}
\end{equation}

\end{proof}

\begin{example}
In Example \ref{C3fl}, it is easy to obtain 
\begin{align}\label{pf}
p_{2*}\shO(h_2+h_3)=N_{\omega}^{\vee}(2h_2)\,\,\,\,\text{and}\,\,\,\, p_{3*}\shO(h_2+h_3)=U_3(2h_3).
\end{align}
Then by Lemma \ref{flip}, we get the flip $g$ (\ref{gflips}) and have the following diagram:
\begin{equation}\label{C3flip}
\begin{tikzcd}
&E=\IFl(2,3,6)  \arrow[dl,swap,"p_2"] \arrow[d, hook, "j"]\arrow[dr,"p_3"]\\
  \IGr(2,6) \arrow[d,hook] & \arrow[dl,swap, "\pi_2"]  X \arrow[dr, "\pi_3"]&\arrow[d, hook',swap] \IGr(3, 6)\\
   X_2 \arrow[rr,dashed, "g"]   &  & X_3.
\end{tikzcd}
\end{equation}
Here $X_2=\PP_{\IGr(2, 6)}(N_{\omega}(-2h_2)\oplus \shO)$ and  $X_3=\PP_{\IGr(3, 6)}(U_3^{\vee}(-2h_3)\oplus \shO).$
\end{example}

\begin{remark}
\begin{enumerate}
\item If $|I|=2$ (in this case $\rho(G/P_I)=2$, and $G/P_J$ is a single point), then $f$ is a \emph{generalized Grassmannian flip}, and the exceptional locus, \emph{i.e.}, the center of the blowing-ups $\pi_{\alpha}$ and $\pi_{\beta}$ are generalized Grassmannians. 
\item Any simple flip of homogeneous type can be viewed as a family of generalized Grassmannian flips over the base $G/P_J$ (see Example \ref{familyC3} below). 
\item It is not hard to see that under \bref{thm:asp}{Condition (A)}, $f$ is a flop iff $\dim\, G/P_{\beta}= \dim\, G/P_{\alpha}$. In this case, we may call it \emph{a simple flop of homogeneous type}, classified by Kanemitsu \cite{kanemitsu2022mukai}.
\end{enumerate}
\end{remark}

We can describe the birational map $g$ explicitly. Note that the fiber of $N_{\omega}$ over $V_2\in \IGr(2, 6)$ can be identified as the two-dimensional vector space $V_2^{\perp}/V_2$, where $V_2^{\perp}=\{v\in \CC^6| \omega (V_2, v)=0\}$. The natural projection to the second factor $E^{\vee}_i\oplus \shO\rightarrow \shO$ defines a divisor $E_i^{\infty}\cong \PP(E_i^{\vee})$ of $X_i$ at \enquote{infinity} ($i=2, 3$), and $X_i^{\circ}:=X_i-E_i^{\infty}=\Tot_{\IGr(i, 6)}(E_i^{\vee})$.  Furthermore, $X_2^{\circ}$ can be identified as
\[
\{(V_2, \alpha)|\, V_2\in \IGr(2, 6), \, \alpha\in V_2^{\perp}/V_2 \otimes (\wedge^2 V_2)^{\otimes 2}=\Hom((\wedge^2 V_2^{\vee})^{\otimes 2}, V_2^{\perp}/V_2)\}.
\]
Similarly, $X_3^{\circ}$ can be identified as
\[
\{(V_3, \beta)|\, V_3\in \IGr(3, 6),\, \beta\in \Hom(V_3, (\wedge^3 V_3)^{\otimes 2})\}.
\]

The birational map $g: X_2\dashrightarrow X_3$ or equivalently $X_2^{\circ} \dashrightarrow X_3^{\circ}$ can be described as follows:
\begin{equation*}
\begin{tikzcd}[%
    ,row sep = 0ex
    ,/tikz/column 1/.append style={anchor=base east}
    ,/tikz/column 2/.append style={anchor=base west}
    ]
 X_2^{\circ}\backslash \IGr(2, 6) \arrow[r,"g", "\sim"'] &X_3^{\circ}\backslash \IGr(3, 6)\\
(V_2, \alpha) \arrow[r, maps to] & (\Im \alpha, \, \beta\colon \Im \alpha \to \Im \alpha/V_2).
\end{tikzcd}
\end{equation*}
First note that $V_2\subsetneq \Im \alpha:=V_3 \subsetneq V_2^{\perp}$. Then $\Im \alpha \in \IGr(3, 6)$, and $\beta$ is the natural quotient.
Next, notice that the symplectic form $\omega$ defines an isomorphism $V_3/V_2\otimes V_2^{\perp}/V_3\cong \CC$, and $\alpha$ defines an isomorphism $V_2^{\perp}/V_3\cong (\wedge^2 V_2)^{\otimes 2}$. Thus, $\beta$ can be identified as a non-zero linear map $V_3\rightarrow V_3/V_2\otimes V_3/V_2\otimes (\wedge^2 V_2)^{\otimes 2}\cong (\wedge^3 V_3)^{\otimes 2}$. Finally, $g$ turns out to be birational by constructing its inverse directly:
\[
(V_3, \beta) \longmapsto (\Ker \beta, \alpha: V_3/\Ker \beta  \rightarrow \Ker \beta^{\perp}/\Ker \beta).
\]

\subsection{Classification of Generalized Grassmannian Flips}
\begin{sloppypar}
By means of Tits shadow, one can check \bref{thm:asp}{Condition (A)} for all Dynkin diagrams with two marked nodes case by case so that we get all generalized Grassmannian flips and flops. Let $I=\{\alpha, \beta\}$, and choose the submarkings $\{\alpha\}$ and $\{\beta\}$. Then Table \ref{cgff} lists all those $(\Gamma, I)$'s that satisfy \bref{thm:asp}{Condition (A)}. 
\end{sloppypar}
 \begin{theorem}\label{main2}
 There is a classification of generalized Grassmannian flips and flops via marked Dynkin diagrams as shown in the table below. Here an ordered pair of nodes $(\alpha, \beta)$ represents the birational map $X_{\alpha}\dashrightarrow X_{\beta}$ in (\ref{pregflip}). (See Appendix for labelings of Dynkin diagrams.)
 \end{theorem}

\begin{table}[H]
\begin{center}
\resizebox{\textwidth}{!}{
\begin{tabular}{|c|c|c|}
\hline
Dynkin diagram   &generalized Grassmannian flips & generalized Grassmannian flops\\
\hline
$A_m \times A_n$ $(m>n)$  &$(\alpha_1, \beta_1)$         &None\\
$A_n \times A_n$  &None  &$(\alpha_1, \beta_1)$        \\
$A_n (n\geq 2)$  &$(\alpha_{i+1}, \alpha_{i})$ $(1\leq i<\dfrac{n}{2})$; $(\alpha_i, \alpha_{i+1})$ $(\dfrac{n}{2}<i<n)$ & $(\alpha_1, \alpha_{n})$ for all $n$ and $(\alpha_{\frac{n}{2}}, \alpha_{\frac{n}{2}+1})$ for even $n$\\
$B_n (n\geq 3)$   & $(\alpha_{n-1}, \alpha_{n})$ for all $n$ and $(\alpha_3, \alpha_1)$ for $n=3$ & None \\
$C_n (n\geq 2)$ &  $(\alpha_{i+1}, \alpha_i) (1\leq i<\dfrac{2n-1}{3})$; $(\alpha_{i}, \alpha_{i+1}) (\dfrac{2n-1}{3}<i<n)$& $(\alpha_{\frac{2n-1}{3}}, \alpha_{\frac{2n-1}{3}+1})$ for $n\equiv2 (\operatorname{mod}\, 3)$\\
$D_n (n\geq 4)$ &None & $(\alpha_{n-1}, \alpha_{n})$\\
$E_n (n=6, 7, 8)$ &None & None \\
$F_4$                &None & $(\alpha_2, \alpha_3)$ \\
$G_2$              & None &$(\alpha_1, \alpha_2)$\\
\hline
\end{tabular}}\caption{Classification of generalized Grassmannian flips and flops}\label{cgff}
\end{center}
\end{table}

\begin{proof}
Since the classification of the generalized Grassmannian flops has been obtained by Kanemitsu \cite{kanemitsu2022mukai} (section 5.2 of loc. cit.), we will focus on the generalized Grassmannian flips. 
\begin{itemize}
\item[(1)] For the $A_m \times A_n$ case, when $m>n$, 
$(\alpha_1, \beta_1)$ gives the standard flip, which is very similar to the $A_n\times A_n$ case in Example 5.1 of  Kanemitsu's paper \cite{kanemitsu2022mukai}. The $m>n$ guarantees the inequality of the dimensions in Proposition \ref{flip}. 
\item[(2)]
The $A_n$ case has been studied in \cite{aflip}. Both conditions $i> n/2$ and $i<n/2$ guarantee the inequality of the dimensions in Proposition \ref{flip}, respectively. 
\item[(3)] The $C_n$ case is similar to the $A_n$ case. Example \ref{C3fl} gives a particular example of this type, and the general $C_n$ case shares a similar pattern of Tits shadow. 
\item[(4)] The $D, E, F$ and $G$ cases have no Tits shadows that have two projective spaces simultaneously except for the flop cases (third column listed in the table). 
\item[(5)] For the $B_n$ case,
$$ \dynkin[edge length=1, root radius=.1cm,labels={,, , , \alpha_{n-1},\alpha_n}]{B}{***.*XX}$$ represents the homogeneous variety $Spin(2n+1, \CC)/P_{n, n-1}$. It has two projections onto:
$$ \dynkin[edge length=1, root radius=.1cm,labels={,, , , \alpha_{n-1},}]{B}{***.*X*}$$ and 
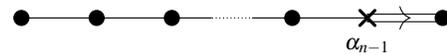
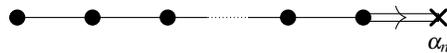
$$\dynkin[edge length=1, root radius=.1cm,labels={,,,,,\alpha_n}]{B}{***.**X}$$ with
Tits shadows:
$$ \dynkin[edge length=1, root radius=.1cm,labels={\alpha_n}]{A}{X}\longleftrightarrow \PP^1$$
and 
$$\dynkin[edge length=1, root radius=.1cm,labels={,,,,\alpha_{n-1}}]{A}{***.*X}\longleftrightarrow \PP^{n-1},\, \text{respectively}.$$
For $n\geq 3$, this gives a simple flip by Proposition \ref{flip}.

It is noted that there is a special flip for type $B_3$ by considering
$$ \dynkin[edge length=1, root radius=.1cm,labels={\alpha_1 ,, \alpha_3}]{B}{X*X},$$
whose Tits shadow onto 
$$ \dynkin[edge length=1, root radius=.1cm,labels={,, \alpha_3}]{B}{X**}$$
and 
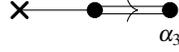
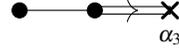
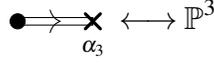
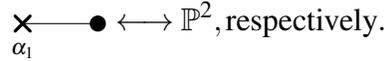
$$ \dynkin[edge length=1, root radius=.1cm,labels={,, \alpha_3}]{B}{**X}$$
are
$$\dynkin[edge length=1, root radius=.1cm,labels={, \alpha_3}]{B}{*X}\longleftrightarrow  \PP^3$$
 and  
$$\dynkin[edge length=1, root radius=.1cm,labels={\alpha_1 ,}]{A}{X*}\longleftrightarrow  \PP^2, \text{respectively}.$$
It is easy to see that $\dim B_3/P_3>\dim B_1/P_1$ and hence, $(\alpha_3, \alpha_1)$ gives a simple flip by Proposition \ref{flip}. 
\end{itemize}

\end{proof}

\begin{example}[(A family of flips (\ref{C3flip}))]\label{familyC3}
Choose the marking $I=\{\alpha_1, \alpha_2, \alpha_3\}$ and the sub-markings $I_1=\{\alpha_1, \alpha_2\}$ and $I_2=\{\alpha_1, \alpha_3\}$:
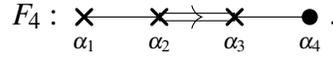
$$F_4: \dynkin[edge length=1, root radius=.1cm,labels={\alpha_1,\alpha_2, \alpha_3, \alpha_4}]{F}{XXX*}.$$
The corresponding simple flip is exactly a family of flips (\ref{C3flip}), \emph{i.e.}, a family of  
$$ C_3:  \dynkin[edge length=1, root radius=.1cm,labels={\alpha_4, \alpha_3,\alpha_2}]{C}{*XX} $$
over the base $B=F_4/P_1$. 
\end{example}

\subsection{DK Conjecture for Simple Flips of Homogeneous Type} 
\begin{conjecture}\label{conjgflip}
In the circumstance of Proposition \ref{flip}, there exists a fully-faithful functor $\Phi$ of triangulated categories
\begin{equation*}
\begin{tikzcd}
 \Phi: D(X_{\beta}) \arrow[r,hook]   & D(X_{\alpha}).
\end{tikzcd}
 \end{equation*}
\end{conjecture}

\subsection*{Strategy for proving Conjecture \ref{conjgflip}}
By Orlov's blow-up formula \cite{orlov1992projective}, we have the following two semi-orthogonal decompositions (SOD) of $D(Y)$:
\begin{align}
&D(Y)=\langle \pi_{\alpha}^*D(Y_{\alpha}), j_*p_{\alpha}^*D(G/P_{\alpha}), \cdots, j_*p_{\alpha}^*D(G/P_{\alpha})\otimes\shO((r_{\alpha}-2)h_{\beta})\rangle; \label{SOD1'}\\
&D(Y)=\langle \pi_{\beta}^*D(Y_{\beta}), j_*p_{\beta}^*D(G/P_{\beta}), \cdots,  j_*p_{\beta}^*D(G/P_{\beta})\otimes\shO((r_{\beta}-2)h_{\alpha})\rangle,\label{SOD2'}
\end{align}
where $r_i=\rank E_i= \rank p_{i*}(\shO(h_{\alpha}+h_{\beta}))$ for $i=\alpha, \beta$.

In principle, we can compare $D(Y_{\alpha})$ and $D(Y_{\beta})$ once both $D(G/P_{\alpha})$ and $D(G/P_{\beta})$ are known in detail. If both $D(G/P_{\alpha})$ and $D(G/P_{\beta})$ admit full exceptional collections, we can apply the mutation techniques of Kuzenetsov (see \cite{kuznetsov2010derived}) to simplify SOD (\ref{SOD1'}) and (\ref{SOD2'}) properly and finally conclude the derived embedding by chess game method. See \cite{bondal1989representation} or the Appendix of \cite{aflip} for details of full exceptional collections, SODs and mutations.

However, it is still open whether $D(G/P_{\alpha})$ admits full exceptional collections in general. Refer to Kapranov \cite{kapranov1988derived}, Kuznetsov  \cite{kuznetsov2008exceptional},  \cite{kuznetsov2016exceptional}, Fonarev \cite{fonarev2019full} and others for constructions of full exceptional collections for certain generalized Grassmannians.

\begin{remark}
\begin{enumerate}
\item  In \cite{aflip}, we proved Conjecture \ref{conjgflip} for the first two nodes $(\alpha_2, \alpha_1)$ of $A_n$ diagram. The case for the first two nodes $(\alpha_2, \alpha_1)$ of $C_n$ diagram can also be proved by modifying the proof in \cite{aflip}. The case for $B_3$ type simple flips will be proved in \cite{bflip} .

\item The embedding result holds as long as the flip behaves like $g$ locally. In other words, if a flip $X'_2 \dashrightarrow X'_3$ satisfies the condition that $ X'_2$ contains $\IGr(2, 6)$ whose normal bundle $N(\IGr(2, 6), X'_2)\cong N_{\omega}(-2h_2)$, then there is an embedding from $D(X'_3)$ into $D(X_2')$
\item By a theorem of Kuznetsov on Hochschild homology (Theorem 7.3 in \cite{kuznetsov2009hochschild}), there is an embedding of vector spaces $HH_{\bullet}(X_3)\hookrightarrow HH_{\bullet}(X_2)$. By the HKR theorem, one immediately gets an embedding (not necessarily graded) on the total cohomology $H(X_3)\hookrightarrow H(X_2)$.
\end{enumerate}
\end{remark}

\section{Computation of mutations}
Let us write down SOD (\ref{SOD1'}) and SOD (\ref{SOD2'}) for our generalized Grassmannian flip for $Sp(6, \CC)$ (\ref{C3flip}):
\begin{align}
&D(X)=\langle \pi_2^*D(X_2), j_*p_2^*D(\IGr(2, 6))\rangle; \label{SOD1}\\
&D(X)=\langle \pi_3^*D(X_3), j_*p_3^*D(\IGr(3, 6)), j_*(p_3^*D(\IGr(3, 6))\otimes\shO(h_2))\rangle. \label{SOD2}
\end{align}

Both $\IGr(2,6)$ and $\IGr(3,6)$ admit full exceptional collections:
\begin{lemma}[(Kuznetsov \cite{kuznetsov2008exceptional} and \cite{kuznetsov2006hyperplane})]\label{SG}
\begin{align}
& D(\IGr(2, 6))=\langle \shA'(-2h_2), \shA'(-h_2), \shA, \shA(h_2), \shA'(2h_2)\rangle,\\
& D(\IGr(3, 6))=\langle \shB(-2h_3), \shB(-h_3), \shB, \shB(h_3)\rangle,
\end{align}
where $\shA=\langle S^2 U_2, U_2, \shO\rangle$, $\shA'=\langle U_2, \shO\rangle$ and $\shB=\langle U_3, \shO\rangle$.
\end{lemma}

For simplicity, we omit the functors $j_*$, $p_2^*$ and $p_3^*$. All relevant sheaves (of complexes) in the following two lemmas live on $X$ by applying pullback to the exceptional set $E$ first and then pushforward to $X$ if not otherwise specified. 
\begin{lemma}\label{orth}
On $X$, we have the following:
\begin{enumerate}
\item $\Ext^{\bullet}(\shO(2h_2), U_2(h_3))=0;$
\item $\Ext^{\bullet}(U_2(2h_2), U_2(h_3))=0;$
\item  $\Ext^{\bullet}(\shO(2h_2), \shO(h_3))=0.$
\end{enumerate}
  \end{lemma}
 \begin{proof}
 \begin{enumerate}
 \item The attaching distinguished triangle associated to $j: \,E\hookrightarrow X$ is given by
\begin{align}
  -\otimes \shO_E(-E)[1]=\shO_E(h_2+h_3)[1]\longrightarrow j^*j_* \longrightarrow id \xrightarrow{[1]}. \label{att}
\end{align}
Applying it to $\shO(2h_2)$ on $E$, we have a distinguished triangle in $D(E)$: 
\begin{align}
  \shO(2h_2)\otimes \shO_E(-E) [1]\longrightarrow j^*j_* \shO(2h_2) \longrightarrow \shO(2h_2) \xrightarrow{[1]}. \label{att1}
\end{align}
\begin{sloppypar}
Now a long exact sequence arises by applying the derived functor $\Ext^{\bullet}(-, U_2(h_3))$ to the distinguished triangle (\ref{att1}): 
\begin{align*}
  \cdots&\rightarrow \Ext^{\bullet}_E(\shO(2h_2), U_2(h_3))\rightarrow \Ext^{\bullet}_E(j^*j_*\shO(2h_2), U_2(h_3)) \\
  &  \rightarrow \Ext^{\bullet-1}_E(\shO(2h_2)\otimes \shO_E(-E), U_2(h_3)) \rightarrow \cdots.
\end{align*}
\end{sloppypar}
Recall that $\shO_E(-E)=\shO_E(h_2+h_3)$ and by adjoint pairs ($j^*, j_*$), we have
\[
  \cdots\rightarrow \Ext^{\bullet}_E(\shO(2h_2), U_2(h_3))\rightarrow \Ext^{\bullet}(\shO(2h_2), U_2(h_3)) \rightarrow  \Ext^{\bullet-1}_E(\shO(3h_2), U_2)\rightarrow   \cdots.
\]
The last term is easy to evaluate: $$\Ext^{\bullet-1}_E(\shO(3h_2), U_2)=\Ext^{\bullet-1}_{\IGr(2, 6)}(\shO(3h_2), U_2)=0$$ by SOD of $D(\IGr(2,6))$ (Lemma \ref{SG}).
For the first term, we use the projection formula (\ref{pf}):
$$\Ext^{\bullet}_E(\shO(2h_2), U_2(h_3))=\Ext^{\bullet}_{\IGr(2, 6)}(U_2^{\vee}(h_2), N_{\omega}^{\vee}).$$
The definition of $N_{\omega}$ gives a short exact sequence on $\IGr(2, 6)$:
\begin{equation}\label{eq:1}
\begin{tikzcd}
0\arrow[r] & U_2 \arrow[r] & Q_4^{\vee} \arrow[r] & N_{\omega}^{\vee} \arrow[r] &0.
\end{tikzcd}
\end{equation}
Here $Q_4=\underline{\CC^6}/U_2$ is the tautological quotient bundle on $\IGr(2, 6)$.   
Therefore, 
$$\Ext^{\bullet}_{\IGr(2, 6)}(U_2^{\vee}(h_2), N_{\omega}^{\vee})=0$$ holds as long as both 
$$\Ext^{\bullet}_{\IGr(2, 6)}(U_2^{\vee}(h_2), U_2)=0\quad  \text{and} \quad \Ext^{\bullet}_{\IGr(2, 6)}(U_2^{\vee}(h_2), Q_4^{\vee})=0 \quad \text{hold}.$$ 
Note that $U_2^{\vee}=U_2(h_2)$. We have 
$$\Ext^{\bullet}_{\IGr(2, 6)}(U_2^{\vee}(h_2), U_2)=\Ext^{\bullet}_{\IGr(2, 6)}(U_2(2h_2), U_2)=0$$
by SOD of $D(\IGr(2,6))$ (Lemma \ref{SG}). Similarly, by SOD of $D(\IGr(2,6))$, we have 
$$\Ext^{\bullet}_{\IGr(2, 6)}(U_2^{\vee}(h_2), U_2^{\vee})=0,\quad \Ext^{\bullet}_{\IGr(2, 6)}(U_2^{\vee}(h_2), \shO)=0,$$
and thus, $\Ext^{\bullet}_{\IGr(2, 6)}(U_2^{\vee}(h_2), Q_4^{\vee})=0.$\footnote{One can also apply the Borel-Weil-Bott theorem for $Sp(6, \CC)$ directly.}
Hence, the middle term $$\Ext^{\bullet}(\shO(2h_2), U_2(h_3))=0.$$

\item Similarly, we have the following long exact sequence by the attaching distinguished triangle (\ref{att}):
\begin{align*}
 \cdots&\rightarrow  \Ext^{\bullet}_E(U_2(2h_2), U_2(h_3))\rightarrow  \Ext^{\bullet}(U_2(2h_2), U_2(h_3)) \\
 &\rightarrow \Ext^{\bullet-1}_E(U_2(3h_2), U_2) \rightarrow \cdots.
\end{align*}
The last term vanishes because
$$\Ext^{\bullet-1}_E(U_2(3h_2), U_2)=\Ext^{\bullet-1}_{\IGr(2, 6)}(U_2(3h_2), U_2)=0$$ by SOD of $D(\IGr(2,6))$ (Lemma \ref{SG}).
By applying the projection formula (\ref{pf}) to the first term, we have
\begin{align*}
\Ext^{\bullet}_E(U_2(2h_2), U_2(h_3))&=\Ext^{\bullet}_{\IGr(2, 6)}(U_2^{\vee}\otimes U_2(h_2), N_{\omega}^{\vee})\\
&=\Ext^{\bullet}_{\IGr(2, 6)}(\shO(h_2)\oplus S^2U_2^{\vee}, N_{\omega}^{\vee}).
\end{align*}
Making use of the short exact sequence (\ref{eq:1}) and SOD of $D(\IGr(2, 6))$ (Lemma \ref{SG}) again, we conclude that $$\Ext^{\bullet}(U_2(2h_2), U_2(h_3))=0.$$

\item We consider $\Ext^{\bullet}_E(\shO(3h_2), \shO)$ and $\Ext^{\bullet}_E(\shO(2h_2), \shO(h_3))$. The former one
 $$\Ext^{\bullet}_E(\shO(3h_2), \shO)=\Ext^{\bullet}_{\IGr(2, 6)}(\shO(3h_2), \shO)=0$$ by SOD of $D(\IGr(2, 6))$ (Lemma \ref{SG}).
The latter one
$$ \Ext^{\bullet}_E(\shO(2h_2),\shO(h_3))=\Ext^{\bullet}_{\IGr(2,6)}(\shO(h_2), N_{\omega}^{\vee})=0$$
by the same argument as before. Combining with the following exact sequence:
\[
 \cdots \rightarrow \Ext^{\bullet}_E(\shO(2h_2), \shO(h_3)) \rightarrow  \Ext^{\bullet}(\shO(2h_2), \shO(h_3))\rightarrow \Ext^{\bullet-1}_E(\shO(3h_2), \shO)\rightarrow \cdots, 
\]
we have $$\Ext^{\bullet}(\shO(2h_2), \shO(h_3))=0.$$
 \end{enumerate}
 \end{proof}

\begin{lemma}\label{mut}
On $X$, we have the following mutations:
 \begin{enumerate}
\item $\LL_{\shO(h_2)} U_2(h_3)=U_3(h_3);$
 \item $\RR_{\shO(h_3)}U_2(2h_2)=U_3^{\vee}(h_2).$
 \end{enumerate}
 \end{lemma}

 \begin{proof}
 \begin{enumerate}
 \item By the attaching distinguished triangle (\ref{att}),
  \begin{align*}
 & \bR Hom(\shO(h_2), U_2(h_3))=\bR Hom_E(\shO(h_2), U_2(h_3))=\bR Hom(U_2^{\vee}, N_{\omega}^{\vee})\\
 &=Cone(\bR Hom(U_2^{\vee}, Q_4[-1])\rightarrow \bR Hom(U_2^{\vee}, U_2^{\vee}[-1]))=\CC[-1].
 \end{align*}
 This corresponds to the Euler sequence on $E=\PP_{\IGr(3, 6)}(U_3^{\vee})$:
 \begin{equation}\label{euler}
\begin{tikzcd}
0\arrow[r] & U_2(h_3) \arrow[r] & U_3(h_3) \arrow[r] & \shO(h_2) \arrow[r] &0.
\end{tikzcd}
\end{equation}
Hence,
 \begin{align*}
& \LL_{\shO(h)} U_2(h_3)=Cone(\bR Hom(\shO(h_2), U_2(h_3))\otimes \shO(h_2)\xrightarrow{ev} U_2(h_3))\\
 &=Cone(\shO(h_2)\rightarrow U_2(h_3)[1])=U_3(h_3).
 \end{align*}
\item By the attaching distinguished triangle (\ref{att}),
\begin{align*}
&\bR Hom(U_2(2h_2), \shO(h_3))=\bR Hom_E(U_2(2h_2), \shO(h_3))=\bR Hom_{\IGr(2, 6)}(U_2(h_2), N_{\omega}^{\vee})\\
&=Cone(\bR Hom(U_2(h_2), U_2)\rightarrow \bR Hom(U_2(h_2), Q_4^{\vee}))=\CC[-1].
\end{align*}
Note that $U_2(2h_2)=U_2^{\vee}(h_2)$, $\CC[-1]$ corresponds to the dual sequence of (\ref{euler}).:
\begin{equation*}
\begin{tikzcd}
0\arrow[r] & \shO(h_3) \arrow[r] & U_3^{\vee}(h_2) \arrow[r] & U_2^{\vee}(h_2) \arrow[r] &0.
\end{tikzcd}
\end{equation*}
 Hence,
\begin{align*}
&\RR_{\shO(H)}U_2(2h_2)=Cone(U_2(2h_2)\xrightarrow{ev^{\vee}}\bR Hom(U_2(2h_2), \shO(h_3))^{\vee}\otimes \shO(h_3))\\
&=Cone(U_2(2h_2)\rightarrow \shO(h_3)[1])=U_3^{\vee}(h_2).
\end{align*}
\end{enumerate}
\end{proof}

\section{Proof of Theorem \ref{thm}}

\subsection{Mutations on SOD (\ref{SOD1})}\mbox{}

\begin{description}[align=left]
 \item [\textbf{Step 1}]
 We right mutate $\pi_2^*D(X_2)$ through $\langle \shA'(-2h_2), \shA'(-h_2)\rangle$ and right mutate\\ 
 $\langle \shA'(-2h_2), \shA'(-h_2)\rangle$ to the far right. Note that $K_{X}|_E=\shO(-h_3-2h_2)$, and we obtain
 \[
 D(X)=\langle \shD_2, \shA, \shA(h_2), \shA'(2h_2), \shA'(h_3), \shA'(h_2+h_3)\rangle,
 \]
 where $\shD_2=\RR_{\langle \shA'(-2h_2), \shA'(-h_2)\rangle}\pi_2^*D(X_2)$.

 \item [Step 2]
 By Lemma \ref{orth}, we can exchange $\shA'(2h_2)$ and $U_2(h_3)$:
 \[
 D(X)=\langle \shD_2, \shA, \shA(h_2), U_2(h_3), \shA'(2h_2), \shO(h_3), \shA'(h_2+h_3)\rangle.
 \]

 \item[Step 3]
 We left mutate $U_2(h_3)$ through $\shO(h_2)$ by Lemma \ref{mut}:
 \[
 D(X)=\langle \shD_2, \shA, S^2U_2(h_2), U_2(h_2),U_3(h_3),\shO(h_2), \shA'(2h_2), \shO(h_3).\shA'(h_2+h_3)\rangle.
\]

\item[Step 4]
Exchange $\shO(2h_2)$ and $\shO(h_3)$ by Lemma \ref{orth}:
\begin{align*}
 D(X)=\langle & \shD_2, \shA, S^2U_2(h_2), U_2(h_2),U_3(h_3),\shO(h_2), U_2(2h_2), \shO(h_3), \\&\shO(2h_2), \shA'(h_2+h_3)\rangle.
\end{align*}

\item[Step 5]
We right mutate $U_2(2h_2)$ through $\shO(h_3)$ and left mutate $U_2(h_2+h_3)$ through $\shO(2h_2)$ by Lemma \ref{mut}:
\begin{align}
 D(X)=\langle & \shD_2, \shA, S^2U_2(h_2), U_2(h_2),U_3(h_3),\shO(h_2), \shO(h_3), U_3^{\vee}(h_2), U_3(h_2+h_3), \label{sod1'}\\
 &\shO(2h_2), \shO(h_2+h_3)\rangle.  \nonumber
\end{align}
\end{description}

\subsection{Mutation on SOD (\ref{SOD2})}
\begin{sloppypar}
We right mutate $\pi_3^*D(X_3)$ through $\langle \shB(-2h_3), \shB(-h_3)\rangle$ and then right mutate $\langle \shB(-2h_3), \shB(-h_3)\rangle$ to the far right, and we obtain
\begin{align}
D(X)=\langle \shD_3, \shB, \shB(h_3), D(\IGr(3,6))(h_2),\shB(-h_3+2h_2),\shB(2h_2)\rangle, \label{sod2'}
\end{align}
where $\shD_3=\RR_{\langle \shB(-2h_3), \shB(-h_3)\rangle}\pi_3^*D(X_3)$.
\end{sloppypar}
\subsection{Comparing SOD (\ref{sod1'}) and (\ref{sod2'})}
To show that there is an embedding $D(X_3)\hookrightarrow D(X_2)$, it is sufficient to show that the composition of the following functors
\begin{center}
\begin{tikzcd}
\phi: \shD_3 \arrow[r, "i_{\shD_3}"] & D(X) \arrow [r, "\pi_{\shD_2}"]  & \shD_2
\end{tikzcd}
 \end{center}
is fully-faithful, where $i_{\shD_3}$ is the natural embedding and $\pi_{\shD_2}$ is left adjoint to the natural embedding $\shD_2 \hookrightarrow D(X)$. That is, for any $x,y \in \shD_3$,
$$ \Hom_{\shD_2}(\phi(x), \phi(y))=\Hom_{\shD_3}(x,y).$$
By adjunction,
\begin{align*}
\Hom(\phi(x), \phi(y))=\Hom(\pi_{\shD_2}i_{\shD_3} x, \pi_{\shD_2}i_{\shD_3}y)=\Hom(i_{\shD_3} x, i_{\shD_2} \pi_{\shD_2}i_{\shD_3}y).
\end{align*}

Therefore, it is sufficient to show that
\begin{equation*}
 \Cone(y\rightarrow \pi_{\shD_2}y)=0.
\end{equation*}
By noticing that $\pi_{\shD_2}=\LL_{{}^{\perp}\shD_2}$, it is sufficient to show that $\Hom(^{{}\perp}\shD_2,y)=0$ holds for any $y\in \shD_3$, which is equivalent to $^{{}\perp}\shD_2\subset\, ^{{}\perp}\shD_3. $

Recall that
\begin{align}
^{{}\perp}\shD_2=\langle & S^2 U_2, U_2, \shO, S^2U_2(h_2), U_2(h_2),U_3(h_3),\shO(h_2), \shO(h_3),\label{SODD2}\\
& U_3^{\vee}(h_2), U_3(h_2+h_3), \shO(2h_2), \shO(h_2+h_3)\rangle. \nonumber
\end{align}
We will show that each SOD component in (\ref{SODD2}) is an exceptional object of $^{{}\perp}\shD_3$. 
\begin{enumerate}
\item[(i)] It is easy to see that 
$$\langle \shO, U_3(h_3), \shO(h_2), \shO(h_3), U_3(h_2+h_3), \shO(2h_2), \shO(h_2+h_3)\rangle \subset\,^{{}\perp}\shD_3.$$
\item[(ii)] By the short exact sequence (\ref{euler}),  $U_2\in \langle U_3, \shO(h_2-h_3)\rangle$. Hence $$\langle U_2, U_2(h_2)\rangle \subset\,^{{}\perp}\shD_3.$$
\item[(iii)] The following Euler sequence on $\IGr(3, 6)$
\begin{equation}
\begin{tikzcd}
0\arrow[r] &U_3 \arrow[r] & \CC^6\otimes \shO \arrow[r] & U_3^{\vee}\arrow[r] &0. 
\end{tikzcd}
\end{equation}\label{euler3}
implies that
 $$U_3^{\vee}(h_2) \in \langle U_3(h_2), \shO(h_2)\rangle\subset\,^{{}\perp}\shD_3.$$
 \item[(iv)] By considering the second symmetric power of  (\ref{euler}):
\begin{equation*}
\begin{tikzcd}
0\arrow[r] & S^2 U_2 \arrow[r] & S^2 U_3 \arrow[r] & U_3(h_2-h_3) \arrow[r] &0
\end{tikzcd}
\end{equation*}
and the Koszul resolution on the Euler sequence (\ref{euler3}) on $\IGr(3, 6)$:
\begin{equation*}
\begin{tikzcd}
0\arrow[r] & S^2 U_3 \arrow[r] & U_3\otimes \CC^6 \arrow[r] & \wedge^2 \CC^6\otimes \shO \arrow[r]& U_3(h_3)\arrow[r] &0,
\end{tikzcd}
\end{equation*}
we can obtain $S^2U_2\subset\, ^{{}\perp}\shD_3$ and $S^2U_2(h_2)\subset\,^{{}\perp}\shD_3.$
 \end{enumerate}
In summary, $\phi$ is fully-faithful and the composition of the following functors gives an embedding:
\begin{equation*}
D(X_3) \xrightarrow{\RR_{\langle \shA'(-2h_2), \shA'(-h_2)\rangle}\circ \pi_3^*} \shD_3\xrightarrow{\phi} \shD_2 \xrightarrow{\pi_{2*}\circ \LL_{\langle \shB(-2h_3), \shB(-h_3)\rangle}}D(X_2), 
\end{equation*}
which completes the proof of Theorem \ref{thm}.

\appendix
\setcounter{secnumdepth}{0}
\section*{Dynkin Diagrams}
We label Dynkin diagrams as follows: 
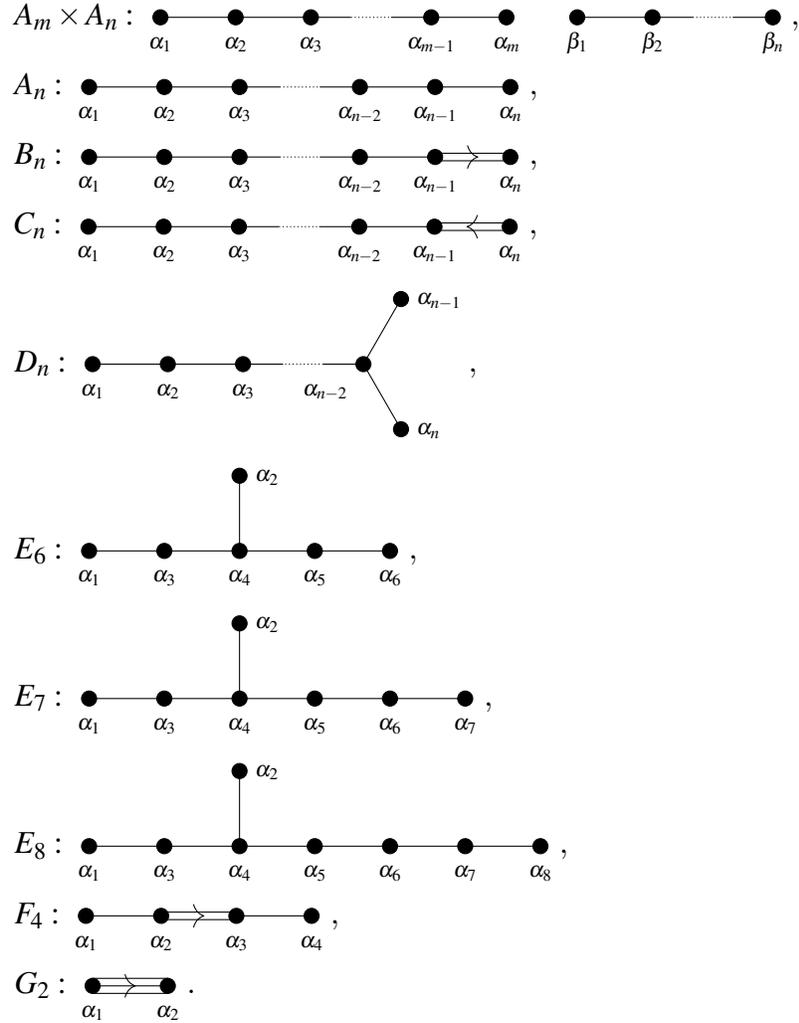
\begin{align*}
 & A_m\times A_n: \dynkin[edge length=1, root radius=.1cm,labels={\alpha_1,\alpha_2, \alpha_3, \alpha_{m-1},\alpha_m}]{A}{***.**} \quad  \dynkin[edge length=1, root radius=.1cm,labels={\beta_1,\beta_2,\beta_n}]{A}{**.*},\\
 & A_n: \dynkin[edge length=1, root radius=.1cm,labels={\alpha_1,\alpha_2, \alpha_3, \alpha_{n-2}, \alpha_{n-1},\alpha_n}]{A}{***.***},\\
 & B_n: \dynkin[edge length=1, root radius=.1cm,labels={\alpha_1,\alpha_2, \alpha_3, \alpha_{n-2}, \alpha_{n-1},\alpha_n}]{B}{***.***},\\
 & C_n: \dynkin[edge length=1, root radius=.1cm,labels={\alpha_1,\alpha_2, \alpha_3,\alpha_{n-2}, \alpha_{n-1},\alpha_n}]{C}{***.***},\\
 & D_n: \dynkin[edge length=1, root radius=.1cm,labels={\alpha_1,\alpha_2, \alpha_3, \alpha_{n-2},\alpha_{n-1},\alpha_n}]{D}{***.***},\\
 & E_6: \dynkin[edge length=1, root radius=.1cm,labels={\alpha_1,\alpha_2, \alpha_3, \alpha_4, \alpha_5, \alpha_6}]{E}{******},\\
 & E_7: \dynkin[edge length=1, root radius=.1cm,labels={\alpha_1,\alpha_2, \alpha_3, \alpha_4, \alpha_5, \alpha_6, \alpha_{7}}]{E}{*******},\\
& E_8: \dynkin[edge length=1, root radius=.1cm,labels={\alpha_1,\alpha_2, \alpha_3, \alpha_4, \alpha_5, \alpha_6, \alpha_{7}, \alpha_8}]{E}{********},\\
 &F_4: \dynkin[edge length=1, root radius=.1cm,labels={\alpha_1,\alpha_2, \alpha_3, \alpha_4}]{F}{****},\\
 & G_2: \dynkin[edge length=1, root radius=.1cm,labels={\alpha_1,\alpha_2}]{G}{**}.
 \end{align*}
 
 \section*{Acknowledgements}
We thank Kowk Wai Chan, Jesse Huang, Lisa Li, Laurent Manivel, Yukinobu Toda,  Zhiwei Zheng, Yan Zhou and especially Yalong Cao, Qingyuan Jiang, Mikhail Kapranov, Yujiro Kawamata, Eduard Looijenga, Chin-Lung Wang for many helpful discussions and suggestions.  
We are grateful for numerous constructive comments from anonymous referees. N. C. L. is supported by grants from the Research Grants Council of the Hong Kong Special Administrative Region, China (Project No. CUHK14301117, CUHK14303518 and Direct grant CUHK4053337).

\end{document}